\documentclass[ngerman,english]{article}

\usepackage[T1]{fontenc}
\usepackage[utf8]{inputenc}
\usepackage{fancyhdr}
\usepackage[pdftex,bookmarks,colorlinks]{hyperref}
\hypersetup{
colorlinks=true,
citecolor=green,
filecolor=magenta,
linkcolor=red,
urlcolor=blue,
}
\usepackage{amsthm}
\usepackage{amsmath}
\usepackage{amssymb}
\theoremstyle{plain}
\newtheorem{thm}{Theorem}

\newtheorem{cor}{Corollary}

\newtheorem{lem}{Lemma}

\newtheorem{tab}{Table} 

\PassOptionsToPackage{normalem}{ulem}
\usepackage{ulem}
\usepackage{amssymb,latexsym,amsmath,epsfig,amsthm} %% Add other packages as necessary

\makeatletter

\renewcommand\section{\@startsection {section}{1}{\z@}
{-30pt \@plus -1ex \@minus -.2ex}
{2.3ex \@plus.2ex}
{\normalfont\normalsize\bfseries}}

\renewcommand\subsection{\@startsection{subsection}{2}{\z@}
{-3.25ex\@plus -1ex \@minus -.2ex}
{1.5ex \@plus .2ex}
{\normalfont\normalsize\bfseries}}

\renewcommand{\@seccntformat}[1]{\csname the#1\endcsname. }

\makeatletter
\providecommand{\tabularnewline}{\\}

\usepackage{fancyhdr}
\PassOptionsToPackage{normalem}{ulem}
\usepackage{ulem}

\newcommand{\noun}[1]{\textsc{#1}}
\providecommand{\tabularnewline}{\\}

\@ifundefined{definecolor}{\usepackage{color}}{}

\makeatother

\usepackage{babel}
\begin{document}
\begin{center}
\uppercase{\bf Derived Ramanujan Primes: $R'_{n}$}
\vskip 20pt
{\bf M. Baris Paksoy}\\
{\it{Humboldt Universit\"{a}t zu Berlin, Berlin, Germany}}\\
{\tt 
\href{mailto:baris.paksoy@gmail.com}{baris.paksoy@gmail.com}}\\ 
\end{center}
%%\vskip 30pt
%%\centerline{\smallit Received: , Revised: , Accepted: , Published: }
\vskip 40pt
\lhead{Derived Ramanujan Primes: $R'_{n}$} \chead{} \rhead{M. Baris
Paksoy}

\selectlanguage{english}%

\centerline{\bf Abstract}
\noindent
In this article, we study the Ramanujan-prime-counting function $\pi_{R}(x)$
along the lines of Ramanujan's original work on Bertrand's Postulate.
We show that the number of Ramanujan primes $R_{n}$ between $x$
and $2x$ tends to infinity with $x$. This analysis leads us to define
a new sequence of prime numbers, which we call derived Ramanujan primes
$R'_{n}$. For $n\geq1$ we define the $n$th derived Ramanujan prime
as the smallest positive integer $R_{n}'$ with the property that if $x\geq R_{n}'$
then $\pi_{R}\left(x\right)-\pi_{R}\left(\frac{x}{2}\right)\geq n$. 
As an application of the existence of derived Ramanujan primes, we prove analogues
for Ramanujan primes of Richert’s Theorem and Greenfield’s Theorem for primes.
We give some new inequalities for both the prime-counting
function $\pi(x)$ and for $\pi_{R}(x)$. Following the recent works
of Sondow and Laishram on the bounds of Ramanujan primes, we analyze
the bounds of derived Ramanujan primes. Finally, we give another proof
of the theorem of Amersi, Beckwith, Miller, Ronan and Sondow, which
states that if $c\in\left(0,1\right)$, then the number of primes
in the interval $\left(cx,x\right)$ tends to infinity with $x$. 
%%\pagestyle{myheadings} 
%%\markright{\smalltt INTEGERS: 12 (2012)\hfill} 
\thispagestyle{empty} 
\baselineskip=12.875pt 
\vskip 40pt
$\hphantom{}$

\subsubsection*{1. Introduction}

\paragraph{\textmd{In 1919 Srinivasa Ramanujan \cite{7} gave an elegant proof
of Bertrand's Postulate, which states that there exists a prime number
between $n$ and $2n$ for all $n\geq2$. In the process he showed
the existence of a certain sequence of prime numbers, now known as
Ramanujan primes. Recall that $\pi\left(x\right)$ is the prime counting
function, that is, $\pi\left(x\right)$ is the number of primes less
than or equal to $x$. In 2009 Jonathan Sondow gave the following
definition in \cite{15}:}\\\\}

\emph{For $n\geq1$, the $n$th Ramanujan prime is the smallest
positive integer R$_{n}$ with the property that if $x\geq R_{n}$,
then $\pi\left(x\right)-\pi\left(\frac{x}{2}\right)\geq n$.}

\paragraph{\textmd{As an example, if }$n=1,\mbox{ }2,\mbox{ }3,\mbox{ }4,\mbox{ }5,\mbox{ }6,\ldots,$
\textmd{then the $n$th Ramanujan prime}
$R_{n}=2,\mbox{ }11,\mbox{ }17,\mbox{ }29,\mbox{ }41,\mbox{ }47,\ldots$ 
$\left(\mbox{\href{http://oeis.org/A104272}{A104272} \textmd{in \cite{14}}}\right)$\textmd{.
After that he proved that the $n$th Ramanujan prime $R_{n}$ lies between
the $2n$th and $4n$th prime for all $n\geq2$. He also showed that
$R_{n}\thicksim p_{2n}$ as $n\to\infty$, and that for every $\varepsilon>0$, there
exists $N_{0}\left(\varepsilon\right)$ such that $R_{n}<\left(2+\varepsilon\right)n\ln n$
for $n\geq N_{0}\left(\varepsilon\right)$. Shanta Laishram in \cite{5}
improved Sondow's result by showing that the $n$th Ramanujan prime
does not exceed the $3n$th prime. In Theorem 1 of \cite{5} Laishram
also gave a method to calculate $N_{0}\left(\varepsilon\right)$.
Following these theorems, we} \textmd{{\it denote by $\pi_{R}\left(x\right)$
the number of Ramanujan primes which do not exceed} $x$ and we show the
existence of derived Ramanujan primes $R'_{n}$ with the similar definition:}\\\\}

\textit{For $n\geq1$, the $n$th derived Ramanujan prime
is the smallest positive integer $R_{n}'$ with the property that if $x\geq R_{n}',$
then $\pi_{R}\left(x\right)-\pi_{R}\left(\frac{x}{2}\right)\geq n$.
In other words, there holds}

\textit{
\[
\pi_{R}\left(x\right)-\pi_{R}\left(\frac{x}{2}\right)\geq1,\,2,\,3,\,4,\,5,\ldots,\: if\: x\geq11,\,41,\,59,\,97,\,149\ldots.
\]
$^{\left(\mbox{\href{http://oeis.org/A192820}{A192820} \cite{14}}\right)}$}

\paragraph{\textmd{Note that the derived Ramanujan primes denoted  in ${\left(\mbox{\href{http://oeis.org/A192820}{A192820} \cite{14}}\right)}$ as 2-Ramanujan primes. The existence of $R'_{n}$ also means that the number of
Ramanujan primes between $x$ and $2x$ tends to infinity with $x$.
This proof makes it possible to give some applications to Ramanujan
primes of Bertrand's Postulate, Richert's Theorem \cite{8} and Greenfield's
Theorem \cite{4} on primes. After that we extend Rosser and Schoenfeld's
Theorem }$2\pi\left(x\right)\geq\pi\left(2x\right)$\textmd{ to the
Ramanujan-prime-counting function $\pi_{R}\left(x\right)$ by proving that }$2\pi_{R}\left(x\right)\geq\pi_{R}\left(2x\right)$\textmd{,
with the help of Segal's idea \cite{11}. This makes it possible to prove that
the $n$th derived Ramanujan prime lies between the $2n$th Ramanujan
prime and the $3n$th Ramanujan prime, and also that $R_{n}'\sim R_{2n}\thicksim p_{4n}$.
In \cite{16} J. Sondow, J. W. Nicholson and T. D. Noe made an analysis
of bounds and runs of Ramanujan primes and showed that if an upper
twin prime is Ramanujan, then so is the lower. In \cite{12} V. Shevelev
studied some parallel properties of Ramanujan primes and Labos primes
and gave generalizations with the construction of two kinds of sieves
for them. Recently, N. Amersi, O. Beckwith, S. J. Miller, R. Ronan
and J. Sondow \cite{1} gave another generalization of Ramanujan primes which states that for any $c\in(0,1)$, the $n$th $c$-Ramanujan
prime can be defined as the smallest integer $R_{c,n}$ such that
for all $x\geq R_{c,n}$, there are at least $n$ primes in the interval
$\left(cx,x\right]$. They also showed that $R_{c,n}\sim p_{\frac{n}{1-c}}$
as $n$ tends to infinity. In the last section we give another proof
of the existence of $c$-Ramanujan primes.}}

\vskip 3pt

\subsubsection*{2. Derived Ramanujan Primes and Two Applications}

$\hphantom{}$

We begin this section with a useful corollary of a theorem of Sondow. Then we show the existence of derived Ramanujan primes and analogues of Richert’s Theorem and Greenfield’s Theorem for Ramanujan primes.

\begin{thm}\label{t1}
\textbf{{\rm (Sondow \cite{15})} }
For every $\varepsilon>0$, there exists $n_{0}=n_{0}\left(\varepsilon\right)$
such that 
\[
R_{n}<\left(2+\varepsilon\right)n\ln n\qquad\left(n\geq n_{0}\right).
\]
\end{thm}

\begin{cor}\label{c1}
For all $n\geq N_{\varepsilon}$ and $x\geq R_{n}$, there hold the inequalities
\begin{equation}
\frac{\pi\left(x\right)}{2}>\pi_{R}\left(x\right)>\frac{\pi\left(x\right)}{2+\varepsilon}.\label{0}
\end{equation}
\end{cor}

\begin{proof}
By Sondow's inequality $R_{n}>p_{2n}$ for $n>1$, the left
side of \eqref{0} must hold because if $\pi_{R}\left(x\right)=n$,
then $\pi\left(x\right)$ must be greater than $2n$. Now we will prove
right side of \eqref{0}. Let $R_{n+1}>x\geq R_{n}$, that is, $\pi_{R}\left(x\right)=n$. It is enough to show that
$\left(2+\varepsilon\right)n>\pi\left(x\right)$. By \autoref{t1} it follows that
\begin{equation}
\pi\left(x\right)<\pi\left(R_{n+1}\right)\leq\pi\left(\left(2+\varepsilon\right)\left(n+1\right)\ln\left(n+1\right)\right).\label{1}
\end{equation}
Now take $\left(2+\varepsilon\right)\left(n+1\right)\ln\left(n+1\right)=k$. For every $\varepsilon$ and $n\geq10$ it is easy to see that the inequality
\begin{equation}
\ln\left(n+1\right)<n\left(\ln\left(k\right)-\ln\left(n+1\right)-1.2762\right)\label{2}
\end{equation}
holds. Hence
\begin{equation}
1.2762n<n\ln k -\left(n+1\right)\ln\left(n+1\right)\label{3}
\end{equation}
and
\begin{equation}
\left(1+\frac{1.2762n}{n\ln k}\right)\left(1-\frac{n\ln\left(k\right)-\left(n+1\right)\ln\left(n+1\right)}{n\ln k}\right)<1. \label{4}
\end{equation}
One can check that \eqref{4} holds for $n\geq5$. As we have $1-\frac{n\ln\left(k\right)-\left(n+1\right)\ln\left(n+1\right)}{n\ln k}=\frac{\left(n+1\right)\ln\left(n+1\right)}{n\ln k}$  we get
\begin{equation}
\frac{k}{\ln k}\left(1+\frac{1.2762}{\ln k}\right)<\left(2+\varepsilon\right)n\label{5}
\end{equation}
and by Dusart's inequality \cite{3} for $x>1$
\begin{equation}
\pi\left(x\right)\leq\frac{x}{\ln x}\left(1+\frac{1.2762}{\ln x}\right)\label{6}
\end{equation}
and \eqref{1}, the inequalities 
\begin{equation}
\pi\left(x\right)<\pi\left(R_{n+1}\right)\leq\pi\left(k\right)<\left(2+\varepsilon\right)n\label{7}
\end{equation}
hold for $n\geq5$, and by computer check also for any $\varepsilon>0$, with $n\geq N_{\varepsilon}$ and {$x\geq R_{n}$}.
\end{proof}

\begin{thm}\label{t2}
There exists at least one Ramanujan prime between $\frac{x}{2}$ and $x$, for all $x\geq11$.
Moreover, the number of Ramanujan primes in the interval
$\left(\frac{x}{2},x\right]$, which is $\pi_{R}\left(x\right)-\pi_{R}\left(\frac{x}{2}\right)$,
tends to infinity with $x$.
\end{thm}

\begin{proof}
By P. Dusart's \cite{3} inequalities 
\begin{equation}
\frac{x}{\ln x}\left(1+\frac{1.2762}{\ln x}\right)\underset{x>1}{>}\pi\left(x\right)\underset{x\geq599}{\geq}\frac{x}{\ln x}\left(1+\frac{1}{\ln x}\right)\label{8}
\end{equation}
and \autoref{c1} we obtain for all $x\geq599$

\begin{equation}
\frac{x}{2\ln x}\left(1+\frac{1.2762}{\ln x}\right)>\frac{\pi\left(x\right)}{2}>\pi_{R}\left(x\right)\geq\frac{\pi\left(x\right)}{3}>\frac{x}{3\ln\left(x\right)}\left(1+\frac{1}{\ln\left(x\right)}\right).\label{9}
\end{equation}
Therefore the inequalities

\begin{equation}
\pi_{R}\left(x\right)-\pi_{R}\left(\frac{x}{2}\right)>\frac{x}{3\ln x}\left(1+\frac{1}{\ln x}\right)-\frac{x}{4\ln\frac{x}{2}}\left(1+\frac{1.2762}{\ln\frac{x}{2}}\right)\geq\frac{x}{\ln x}\left(\frac{1}{12}-\frac{0.3}{\ln x}\right)\label{10}
\end{equation}
hold for all $x\geq599$, where the right side of the last inequality tends
to infinity with $x$. To verify that there exists
at least one Ramanujan prime between \emph{$\frac{x}{2}$ }and\emph{
$x$} for $11\leq x\leq599$, it is enough to see that there exists
one of the Ramanujan primes $11$, $17$, $29$, $47$, $71$, $127$, $241$ and $461$
between $\frac{x}{2}$ and $x$. 
\end{proof}

$\hphantom{}$

Since $\pi_{R}\left(x\right)-\pi_{R}\left(\frac{x}{2}\right)$ is
greater than the monotone increasing function in \eqref{10}, the number of Ramanujan primes between $\frac{x}{2}$ and $x$ tends
to infinity with~$x$. As a result, derived Ramanujan primes exist.

$\hphantom{}$

\begin{center}
\begin{tabular}{|c|c|c|c|c|c|c|c|c|c|c|c|c|c|}
\cline{1-2} \cline{4-5} \cline{7-8} \cline{10-11} \cline{13-14} 
$n$  & $R_{n}'$  &  & $n$  & $R_{n}'$  &  & $n$  & $R_{n}'$  &  & $n$  & $R_{n}'$  &  & $n$  & $R_{n}'$\tabularnewline
\cline{1-2} \cline{4-5} \cline{7-8} \cline{10-11} \cline{13-14} 
1  & 11  &  & 11  & 263  &  & 21  & 599  &  & 31  & 1009  &  & 41  & 1373\tabularnewline
\cline{1-2} \cline{4-5} \cline{7-8} \cline{10-11} \cline{13-14} 
2  & 41  &  & 12  & 307  &  & 22  & 641  &  & 32  & 1019  &  & 42  & 1423\tabularnewline
\cline{1-2} \cline{4-5} \cline{7-8} \cline{10-11} \cline{13-14} 
3  & 59  &  & 13  & 367  &  & 23  & 643  &  & 33  & 1021  &  & 43  & 1427\tabularnewline
\cline{1-2} \cline{4-5} \cline{7-8} \cline{10-11} \cline{13-14} 
4  & 97  &  & 14  & 373  &  & 24  & 647  &  & 34  & 1031  &  & 44  & 1439\tabularnewline
\cline{1-2} \cline{4-5} \cline{7-8} \cline{10-11} \cline{13-14} 
5  & 149  &  & 15  & 401  &  & 25  & 653  &  & 35  & 1049  &  & 45  & 1481\tabularnewline
\cline{1-2} \cline{4-5} \cline{7-8} \cline{10-11} \cline{13-14} 
6  & 151  &  & 16  & 409  &  & 26  & 719  &  & 36  & 1051  &  & 46  & 1487\tabularnewline
\cline{1-2} \cline{4-5} \cline{7-8} \cline{10-11} \cline{13-14} 
7  & 227  &  & 17  & 569  &  & 27  & 751  &  & 37  & 1061  &  & 47  & 1549\tabularnewline
\cline{1-2} \cline{4-5} \cline{7-8} \cline{10-11} \cline{13-14} 
8  & 229  &  & 18  & 571  &  & 28  & 821  &  & 38  & 1063  &  & 48  & 1553\tabularnewline
\cline{1-2} \cline{4-5} \cline{7-8} \cline{10-11} \cline{13-14} 
9  & 233  &  & 19  & 587  &  & 29  & 937  &  & 39  & 1217  &  & 49  & 1559\tabularnewline
\cline{1-2} \cline{4-5} \cline{7-8} \cline{10-11} \cline{13-14} 
10  & 239  &  & 20  & 593  &  & 30  & 941  &  & 40  & 1367  &  & 50  & 1567\tabularnewline
\cline{1-2} \cline{4-5} \cline{7-8} \cline{10-11} \cline{13-14} 
\end{tabular}
\par\end{center}

\begin{tab}
\noun{The First 50 Derived Raman}\emph{\noun{ujan Primes}} 
\end{tab}

In 1948 Hans-Egon Richert \cite{8} proved that each natural number $n\geq7$
can be expressed as a sum of distinct primes. His method has been generalized by Sierpinski, who showed the following theorem.

$\hphantom{}$

\begin{thm}\label{t3}
\textbf{ {\rm (Sierpinski \cite{13})} }
Let $m_{1}$, $m_{2},\ldots$
be an infinite increasing sequence of natural numbers such that for
a certain natural number $k$ the inequality 
\begin{equation}
m_{i+1}\leq2m_{i}\mbox{ }\mbox{ for }\mbox{ }i>k\label{11}
\end{equation}
 holds. If there exists an integer $a\geq0$ and natural numbers
$r$ and $s_{r-1}\geq m_{k+r}$ such that each of the numbers 
\[
a+1\mbox{, }\quad a+2\mbox{, }\ldots\mbox{ , }\quad a+s_{r-1}
\]
 is the sum of different numbers of the sequence $m_{1}$, $m_{2},\ldots,
m_{k+r-1}$, then for $s{}_{r}=s_{r-1}+m_{k+r}$ each of the numbers
\[
a+1\mbox{, }\quad a+2\mbox{, }\ldots\mbox{ , }\quad a+s_{r}
\]
 is the sum of different numbers of the sequence $m_{1},m_{2},\ldots,
m_{k+r}$, and moreover $s_{r}\geq m_{k+r+1}$.
\end{thm}

$\hphantom{}$

\begin{cor}\label{c2}
Each natural number $n\geq123$ can be
expressed as a sum of distinct Ramanujan primes.
\end{cor}

\begin{proof} Let $m_{i}=R_{i}$, $k=0$, $r=10$, $a=122$ and $s_{9}=97$.
There exists at least one Ramanujan prime between $x$ and $2x$ for $x\geq11$ by \autoref{t2}. So we get $R_{i}<R_{i+1}<2R_{i}$ for all natural numbers
$i\geq2$ and this implies the condition \eqref{11}. From \autoref{table2} it can be seen that
 each number from $123$ to $224$ is the sum of different Ramanujan
primes $R_{1}$, $R_{2}$, $\ldots$, $R_{9}$. So each
natural number greater than $123$ can be expressed as a sum of distinct Ramanujan primes.
\end{proof}

\begin{center}
\begin{tabular}{|c|c||c|c||c|c|}
\hline 
a$+$j  & Expression  & a$+$j  & Expression  & a$+$j  & Expression\tabularnewline
\hline 
\hline 
123  & $71+41+11$  & 157  & $71+67+17+2$  & 191  & $101+71+17+2$\tabularnewline
\hline 
124  & $67+29+17+11$  & 158  & $71+59+17+11$  & 192  & $71+67+41+11+2$\tabularnewline
\hline 
125  & $71+41+11+2$  & 159  & $71+59+29$  & 193  & $59+47+41+29+17$\tabularnewline
\hline 
126  & $67+59$  & 160  & $71+59+17+11+2$  & 194  & $71+59+47+17$\tabularnewline
\hline 
127  & $67+47+11+2$  & 161  & $71+59+29+2$  & 195  & $71+67+29+17+11$\tabularnewline
\hline 
128  & $71+29+17+11$  & 162  & $67+47+29+17+2$  & 196  & $71+67+47+11$\tabularnewline
\hline 
129  & $71+47+11$  & 163  & $59+47+29+17+11$  & 197  & $71+67+59$\tabularnewline
\hline 
130  & $71+59$  & 164  & $71+47+29+17$  & 198  & $71+67+47+11+2$\tabularnewline
\hline 
131  & $67+47+17$  & 165  & $59+47+29+17+11+2$  & 199  & $71+67+59+2$\tabularnewline
\hline 
132  & $71+59+2$  & 166  & $71+67+17+11$  & 200  & $71+59+41+29$\tabularnewline
\hline 
133  & $67+47+17+2$  & 167  & $71+67+29$  & 201  & $71+59+41+17+11+2$\tabularnewline
\hline 
134  & $59+47+17+11$  & 168  & $71+67+17+11+2$  & 202  & $71+67+47+17$\tabularnewline
\hline 
135  & $71+47+17$  & 169  & $71+67+29+2$  & 203  & $67+59+47+17+11+2$\tabularnewline
\hline 
136  & $59+47+17+11+2$  & 170  & $71+59+29+11$  & 204  & $71+67+47+17+2$\tabularnewline
\hline 
137  & $71+47+17+2$  & 171  & $71+59+41$  & 205  & $71+59+47+17+11$\tabularnewline
\hline 
138  & $71+67$  & 172  & $71+59+29+11+2$  & 206  & $71+59+47+29$\tabularnewline
\hline 
139  & $67+59+11+2$  & 173  & $71+59+41+2$  & 207  & $71+67+41+17+11$\tabularnewline
\hline 
140  & $71+67+2$  & 174  & $67+59+29+17+2$  & 208  & $71+67+59+11$\tabularnewline
\hline 
141  & $71+59+11$  & 175  & $67+59+47+2$  & 209  & $71+67+41+17+11+2$\tabularnewline
\hline 
142  & $71+41+17+11+2$  & 176  & $59+47+41+29$  & 210  & $71+67+41+29+2$\tabularnewline
\hline 
143  & $71+59+11+2$  & 177  & $59+47+41+17+11+2$  & 211  & $71+59+41+29+11$\tabularnewline
\hline 
144  & $67+47+17+11+2$  & 178  & $71+67+29+11$  & 212  & $67+47+41+29+17+11$\tabularnewline
\hline 
145  & $67+59+17+2$  & 179  & $71+67+41$  & 213  & $71+67+47+17+11$\tabularnewline
\hline 
146  & $71+47+17+11$  & 180  & $71+67+29+11+2$  & 214  & $71+67+59+17$\tabularnewline
\hline 
147  & $71+59+17$  & 181  & $71+67+41+2$  & 215  & $71+67+47+17+11+2$\tabularnewline
\hline 
148  & $71+47+17+11+2$  & 182  & $71+59+41+11$  & 216  & $71+67+59+17+2$\tabularnewline
\hline 
149  & $71+67+11$  & 183  & $67+59+29+17+11$  & 217  & $71+59+47+29+11$\tabularnewline
\hline 
150  & $67+41+29+11+2$  & 184  & $71+67+29+17$  & 218  & $71+59+47+41$\tabularnewline
\hline 
151  & $71+67+11+2$  & 185  & $71+67+47$  & 219  & $71+67+41+29+11$\tabularnewline
\hline 
152  & $71+41+29+11$  & 186  & $71+67+29+17+2$  & 220  & $71+59+47+41+2$\tabularnewline
\hline 
153  & $101+41+11$  & 187  & $71+67+47+2$  & 221  & $71+67+41+29+11+2$\tabularnewline
\hline 
154  & $71+41+29+11+2$  & 188  & $71+59+47+11$  & 222  & $97+71+41+11+2$\tabularnewline
\hline 
155  & $71+67+17$  & 189  & $71+59+29+17+11+2$  & 223  & $71+59+47+29+17$\tabularnewline
\hline 
156  & $67+59+17+11+2$  & 190  & $71+67+41+11$  & 224  & $67+59+41+29+17+11$\tabularnewline
\hline 
\end{tabular}
\par\end{center}

\begin{tab}
\noun{Expressıons of Natural Numbers between 123 and 224
as Sums of Different Ramanujan Primes\label{table2}} 
\end{tab}

In \cite{4}  L. Greenfield and S. Greenfield showed that the integers\emph{
}$\left\{ 1,\mbox{ }2,\mbox{ }\ldots,\mbox{ }2k\right\} $ can be
arranged in $k$ disjoint pairs such that the sum of the elements
in each pair is prime. Similar result can be shown for Ramanujan primes
with their method.

$\hphantom{}$

\begin{cor}\label{c3}
For all integers $k\geq17$ the numbers
$\left\{ 1,\mbox{ }2,\mbox{ }\ldots,\mbox{ }2k\right\} $ can be arranged
in $k$ disjoint pairs such that the sum of the elements in each pair
is a Ramanujan prime.
\end{cor}
\vskip 10pt

\begin{proof} From \autoref{table3} it can be seen for $k=17$ that our assumption
is true. There exists at least one Ramanujan prime between $2k$ and
$4k$ for $k\geq3$ by \autoref{t2}. Now let $j\geq17$ and $2k+j$ be a Ramanujan
prime. Therefore $\left\{ j,\mbox{ }j+1,\mbox{ }\ldots,\mbox{ }2k-1,\mbox{ }2k\right\} $
can be paired as sum of each pair will be equal to $2k+j$, namely
$\left\{ j,2k\right\} $, $\left\{ j+1,2k-1\right\} $, $\left\{ j+2,2k-2\right\} $,
..., $\left\{ \left\lfloor \frac{j+2k}{2}\right\rfloor ,\left\lfloor \frac{j+2k}{2}\right\rfloor +1\right\} $.
Also, by induction $\left\{ 1,\mbox{ }2,\mbox{ }\ldots,\mbox{ }j-1\right\} $
can be arranged in disjoint pairs if $j-1\geq34$. So it is enough
to show that we can always find such an odd natural number $j$ or
equivalently that there exist a Ramanujan prime in the interval $\left(2k+34,4k\right)$. One can easily check that $\left\{ 1,\mbox{ }2,\mbox{ }\ldots,\mbox{ }2k\right\}$ can be arranged in $k$ disjoint pairs as $k\leq17$ only for $k\in \left\{ 5, 6, 8, 9, 11, 12, 14, 15, 17\right\}$. Some certain arrangements given in \autoref{table3}. From \autoref{table3} it can be seen if $j-1\in M=\left\{ 10,\mbox{ }12,\mbox{ }16,\mbox{ }18,\mbox{ }22,\mbox{ }24,\mbox{ }28,\mbox{ }30\right\} $ or by induction hypothesis if $\geq34$ that there is a way to pair the
set. So there is no solution if and only if $j-1\in N=\left\{ 2,\mbox{ }4,\mbox{ }6,\mbox{ }8,\mbox{ }14,\mbox{ }20,\mbox{ }26,\mbox{ }32\right\} $.
But as $R_{9}\left(2\right)=233$ there must be least $9$ choices
for $j$ if $k\geq117$. So all solutions can not be from $N$. By
\autoref{table3} our statement is also verified for $17\leq k<117$.
\end{proof}

\vskip 29pt

\begin{center}
\begin{tabular}{ccc|c|c|c|c|c|c|c|}
\hline 
\multicolumn{1}{|c|}{$k$} & \multicolumn{1}{c|}{{\small $5$}} & {\small $6$}  & {\small $8$}  & {\small $9$}  & {\small $11$}  & {\small $12$}  & {\small $14$}  & {\small $15$}  & {\small $17$}\tabularnewline
\hline 
\multicolumn{1}{c|}{} & \multicolumn{1}{c|}{{\small $1,10$}} & {\small $1,10$}  & {\small $1,16$}  & {\small $1,10$}  & {\small $1,10$}  & {\small $1,10$}  & {\small $1,10$}  & {\small $1,10$}  & {\small $1,10$}\tabularnewline
\cline{2-10} 
\multicolumn{1}{c|}{} & \multicolumn{1}{c|}{{\small $2,9$}} & {\small $2,9$}  & {\small $2,15$}  & {\small $2,9$}  & {\small $2,9$}  & {\small $2,9$}  & {\small $2,9$}  & {\small $2,9$}  & {\small $2,9$}\tabularnewline
\cline{2-10} 
\multicolumn{1}{c|}{} & \multicolumn{1}{c|}{{\small $3,8$}} & {\small $3,8$}  & {\small $3,14$}  & {\small $3,8$}  & {\small $3,8$}  & {\small $3,8$}  & {\small $3,8$}  & {\small $3,8$}  & {\small $3,8$}\tabularnewline
\cline{2-10} 
\multicolumn{1}{c|}{} & \multicolumn{1}{c|}{{\small $4,7$}} & {\small $4,7$}  & {\small $4,13$}  & {\small $4,7$}  & {\small $4,7$}  & {\small $4,7$}  & {\small $4,7$}  & {\small $4,7$}  & {\small $4,7$}\tabularnewline
\cline{2-10} 
\multicolumn{1}{c|}{} & \multicolumn{1}{c|}{{\small $5,6$}} & {\small $5,12$}  & {\small $5,12$}  & {\small $5,6$}  & {\small $5,6$}  & {\small $5,6$}  & {\small $5,6$}  & {\small $5,6$}  & {\small $5,6$}\tabularnewline
\cline{2-10} 
 & \multicolumn{1}{c|}{} & {\small $6,11$}  & {\small $6,11$}  & {\small $11,18$}  & {\small $11,18$}  & {\small $11,18$}  & {\small $11,18$}  & {\small $11,18$}  & {\small $11,18$}\tabularnewline
\cline{3-10} 
 &  &  & {\small $7,10$}  & {\small $12,17$}  & {\small $12,17$}  & {\small $12,17$}  & {\small $12,17$}  & {\small $12,17$}  & {\small $12,17$}\tabularnewline
\cline{4-10} 
 &  &  & {\small $8,9$}  & {\small $13,16$}  & {\small $13,16$}  & {\small $13,16$}  & {\small $13,16$}  & {\small $13,16$}  & {\small $13,16$}\tabularnewline
\cline{4-10} 
 &  & \multicolumn{1}{c}{} &  & {\small $14,15$}  & {\small $14,15$}  & {\small $14,15$}  & {\small $14,15$}  & {\small $14,15$}  & {\small $14,15$}\tabularnewline
\cline{5-10} 
 &  & \multicolumn{1}{c}{} & \multicolumn{1}{c}{} &  & {\small $19,22$}  & {\small $19,22$}  & {\small $19,28$}  & {\small $19,28$}  & {\small $19,22$}\tabularnewline
\cline{6-10} 
 &  & \multicolumn{1}{c}{} & \multicolumn{1}{c}{} &  & {\small $20,21$}  & {\small $20,21$}  & {\small $20,27$}  & {\small $20,27$}  & {\small $20,21$}\tabularnewline
\cline{6-10} 
 &  & \multicolumn{1}{c}{} & \multicolumn{1}{c}{} & \multicolumn{1}{c}{} &  & {\small $23,24$}  & {\small $21,26$}  & {\small $21,26$}  & {\small $23,24$}\tabularnewline
\cline{7-10} 
 &  & \multicolumn{1}{c}{} & \multicolumn{1}{c}{} & \multicolumn{1}{c}{} & \multicolumn{1}{c}{} &  & {\small $22,25$}  & {\small $22,25$}  & {\small $25,34$}\tabularnewline
\cline{8-10} 
 &  & \multicolumn{1}{c}{} & \multicolumn{1}{c}{} & \multicolumn{1}{c}{} & \multicolumn{1}{c}{} &  & {\small $23,24$}  & {\small $23,24$}  & {\small $26,33$}\tabularnewline
\cline{8-10} 
 &  & \multicolumn{1}{c}{} & \multicolumn{1}{c}{} & \multicolumn{1}{c}{} & \multicolumn{1}{c}{} & \multicolumn{1}{c}{} &  & {\small $29,30$}  & {\small $27,32$}\tabularnewline
\cline{9-10} 
 &  & \multicolumn{1}{c}{} & \multicolumn{1}{c}{} & \multicolumn{1}{c}{} & \multicolumn{1}{c}{} & \multicolumn{1}{c}{} & \multicolumn{1}{c}{} &  & {\small $28,31$}\tabularnewline
\cline{10-10} 
 &  & \multicolumn{1}{c}{} & \multicolumn{1}{c}{} & \multicolumn{1}{c}{} & \multicolumn{1}{c}{} & \multicolumn{1}{c}{} & \multicolumn{1}{c}{} &  & {\small $29,30$}\tabularnewline
\cline{10-10} 
\end{tabular}
\par\end{center}

\begin{tab}
\noun{Partitions of Sets $\left\{ 1,2,\ldots,2k\right\} $
for Certain Numbers $k$ up to 17\label{table3}} 
\end{tab}

\subsubsection*{3. Some Inequalities for $\pi\left(x\right)$}

$\hphantom{}$

In this section we prove some inequalities for the prime-counting function by using Dusart\rq{}s inequalities to show $2\pi_{R}\left(x\right)>\pi_{R}\left(2x\right)$ and get better bounds for derived Ramanujan primes.

\begin{lem}\label{l1}
For $x\geq569$ the inequality 
\[
\pi\left(2x\right)-\pi\left(x\right)\leq2\left(\pi\left(x\right)-\pi\left(\frac{x}{2}\right)\right)
\]
 holds.
 \end{lem}

\begin{proof} By P. Dusart's \cite{3} inequalities

\begin{equation}
\frac{x}{\ln x-1}\underset{x\geq5393}{\leq}\pi\left(x\right)\underset{x\geq60184}{\leq}\frac{x}{\ln x-1.1}\label{12}
\end{equation}
it is enough to show that 
\begin{equation}
\frac{x}{\ln\frac{x}{2}-1.1}+\frac{2x}{\ln2x-1.1}\leq\frac{3x}{\ln x-1}.\label{13}
\end{equation}
Therefore we deduce that 
\begin{equation}
\frac{x}{\ln\frac{x}{2}-1.1}+\frac{2x}{\ln2x-1.1}\leq\frac{x}{\ln x-1.8}+\frac{2x}{\ln x-0.41}=\frac{3x\ln x-4.01x}{\ln^{2}x-2.21\ln x+0.738}\label{14}
\end{equation}
and for $x\geq\exp4.72631\geq112.877$ 
\begin{equation}
\frac{3x\ln x-4.01x}{\ln^{2}x-2.21\ln x+0.738}\leq\frac{3x}{\ln x-1}.\label{15}
\end{equation}
By computer check we also verify our statement for $569\leq x\leq60184$.
\end{proof}

$\hphantom{}$

In \cite{10} Rosser and Schoenfeld showed that \textit{\emph{for}}\textit{
}\textit{\emph{$x\geq20.5$ the inequality $\pi\left(2x\right)-\pi\left(x\right)>\frac{3}{5}\frac{x}{\ln x}$
holds. In \cite{3} Dusart improved this result and showed that }}the
inequality\emph{ $\pi\left(2x\right)-\pi\left(x\right)>\frac{x}{\ln x}-\frac{0.7x}{\ln^{2}x}$
}holds for \textit{\emph{$x\geq1328.5$}}.\textit{\emph{ In \cite{5}
Laishram showed that }}\emph{$\pi\left(x\right)-\pi\left(\frac{x}{2}\right)>\frac{x}{2\ln x}-\frac{0.010182x}{\ln^{2}x}$
}\textit{\emph{for $x\geq21088222$ by using Dusart's inequality $\left|\vartheta\left(x\right)-x\right|\leq\frac{0.006788x}{\ln x}$,
where $\vartheta\left(x\right)$ denotes Chebyshev function, equal
to $\underset{p\leq x}{\sum}\ln p$. In \cite{3} Dusart
gave better inequalities for $\vartheta\left(x\right)$. Following Laishram's
proof we will improve the bound for $\pi\left(x\right)-\pi\left(\frac{x}{2}\right)$
to get a better bound in \autoref{l4}.}}

$\hphantom{}$

\begin{lem}\label{l2}
For any $x\geq75374781$ the inequality

\emph{
\[
\pi\left(x\right)-\pi\left(\frac{x}{2}\right)>\frac{x}{2\ln x}\left(1-\frac{31.24}{\ln^{3}x}\right)
\]
 }

holds.
\end{lem}

\begin{proof} By P. Dusart's \cite{3} inequality 
\begin{equation}
\left|\vartheta\left(x\right)-x\right|\leq\frac{10x}{\ln^{3}x}\label{16}
\end{equation}
 for any $x\geq32321$ we get 
\begin{equation}
\pi\left(x\right)-\pi\left(\frac{x}{2}\right)\geq\frac{\vartheta\left(x\right)-\vartheta\left(\frac{x}{2}\right)}{\ln x}\geq\frac{x}{\ln x}\left(1-\frac{10}{\ln^{3}x}-\frac{1}{2}\left(1+\frac{10}{\ln^{3}\frac{x}{2}}\right)\right)\label{17}
\end{equation}
and for $x\geq75374781$ 
\begin{equation}
\geq\frac{x}{2\ln x}\left(1-\frac{31.24}{\ln^{3}x}\right)\label{18}
\end{equation}
holds.
\end{proof}

$\hphantom{}$

\subsubsection*{4. Bounds for Derived Ramanujan Primes}

$\hphantom{}$

To prove a similar result to J. B. Rosser and L. Schoenfeld's inequality
\cite{9} $2\pi\left(x\right)>\pi\left(2x\right)$ for
Ramanujan primes, namely, $2\pi_{R}\left(x\right)>\pi_{R}\left(2x\right)$,
we will use the idea of S. L. Segal \cite{11}.

$\hphantom{}$

\begin{lem}\label{l3}
Let $k$ and $l$ be positive integers. The following two conditions are equivalent:

(i) $R_{k}+R_{l}\leq R_{k+l-1}.$

(ii) If $R_{k-1}\leq x<R_{k}$ and $R_{l-1}\leq y<R_{l}$, then the
inequality
\[
\pi_{R}\left(x+y\right)\leq\pi_{R}\left(x\right)+\pi_{R}\left(y\right)
\]
 holds.
 \end{lem}

\begin{proof} \textbf{(i)}$\Rightarrow$\textbf{(ii):} By the conditions
on $x$ and $y$ it is easy to see that $x+y<R_{k}+R_{l}$ and $\pi_{R}\left(x+y\right)\leq\pi_{R}\left(R_{k}+R_{l}-1\right)$.
Likewise, one can check that

\begin{equation}
\pi_{R}\left(R_{k+l-2}\right)=k+l-2=\pi_{R}\left(R_{k-1}\right)+\pi_{R}\left(R_{l-1}\right)\leq\pi_{R}\left(x\right)+\pi_{R}\left(y\right).\label{19}
\end{equation}
 By (i) we get $R_{k}+R_{l}-1\leq R_{k+l-1}-1$ and easily 
\begin{equation}
\pi_{R}\left(x+y\right)\leq\pi_{R}\left(R_{k}+R_{l}-1\right)\leq\pi_{R}\left(R_{k+l-1}-1\right)\label{20}
\end{equation}
\begin{equation}
=\pi_{R}\left(R_{k+l-2}\right)\leq\pi_{R}\left(x\right)+\pi_{R}\left(y\right).\label{21}
\end{equation}

$\hphantom{}$

\textbf{(ii)}$\Rightarrow$\textbf{(i):} Set $x=R_{k}-\frac{1}{2}$
and $y=R_{l}-\frac{1}{2}$. Therefore we get $\pi_{R}\left(x\right)+\pi_{R}\left(y\right)=k+l-2$
and $\pi_{R}\left(x+y\right)=\pi_{R}\left(R_{k}+R_{l}-1\right)$.
By (ii) we deduce that $k+l-2\geq\pi_{R}\left(R_{k}+R_{l}-1\right)$
and $R_{k+l-1}-1\geq R_{k}+R_{l}-1$.
\end{proof}

$\hphantom{}$

\begin{thm}\label{t4}
For $x\geq11$ the inequality

\[
2\pi_{R}\left(x\right)>\pi_{R}\left(2x\right)
\]

 holds.
 \end{thm}

\begin{proof} By \autoref{l3} it is enough to show $2R_{n}\leq R_{2n-1}$.
But that is equivalent to $\pi_{R}\left(2R_{n}-1\right)\leq\pi_{R}\left(R_{2n}-1\right)$,
i.e., $2R_{n}\leq R_{2n}$. There we will use the idea of the proof
of Theorem 2 in \cite{15} and we will show that the inequality 
\begin{equation}
\pi\left(2R_{n}\right)-\pi\left(R_{n}\right)\leq2n\label{22}
\end{equation}
 holds. By \autoref{l1} we easily deduce that 
\begin{equation}
\pi\left(2R_{n}\right)-\pi\left(R_{n}\right)\leq2\left(\pi\left(R_{n}\right)-\pi\left(\frac{R_{n}}{2}\right)\right)=2n.\label{23}
\qedhere\end{equation}
\end{proof}

\begin{lem}\label{l4}
The $n$th Ramanujan prime satisfies the inequality

\[
R_{n}<\frac{8}{3}n\ln n
\]

for any $n\geq5315$.
\end{lem}

\begin{proof} It is enough to show that $\pi\left(x\right)-\pi\left(\frac{x}{2}\right)>n$
if $x\geq\frac{8}{3}n\ln n$. We have
\begin{equation}
\frac{x}{\ln x}\geq\frac{8n\ln n}{3\ln\left(\frac{8}{3}n\ln n\right)}>2.011n\label{24}
\end{equation}
 for all $n\geq2193650$. By \autoref{l2} we deduce that 
\begin{equation}
\pi\left(x\right)-\pi\left(\frac{x}{2}\right)\geq\frac{x}{2\ln x}\left(1-\frac{31.24}{\ln^{3}x}\right)\geq1.0055n\left(1-\frac{31.24}{\ln^{3}x}\right)\label{25}
\end{equation}
where $1-\frac{31.24}{\ln^{3}x}>\frac{1}{1.0055}$ for $x\geq75374781$.
As $\pi\left(x\right)-\pi\left(\frac{x}{2}\right)>n$ for $x\geq75374781$
and $R_{2113924}=75374791$, we may take $R_{m+1}>x\geq R_{m}$ for
$m\geq2113924$. So our statement is true for $n\geq2113924$. By
computer check we see that our statement is also true for $5315\leq n<2113924$.
\end{proof}

\begin{thm}\label{t5}
The $n$th derived Ramanujan prime satisfies the inequalities
\begin{equation}
R_{2n}\leq R_{n}'<R_{3n}\label{a}
\end{equation}
for any $n\geq1$.
\end{thm}

\begin{proof} For $n=1$, the inequalities hold. If $n>1$, to prove
the left side of \eqref{a}, it is enough to show that $\pi_{R}\left(R_{2n}\right)-\pi_{R}\left(\frac{R_{2n}}{2}\right)\leq n$.
By \autoref{t4} we can see that 
\begin{equation}
2\pi_{R}\left(\frac{R_{2n}}{2}\right)\geq\pi_{R}\left(R_{2n}\right)=2n\label{26}
\end{equation}
 holds. As $R_{2\cdot1}=11$, by \eqref{26} the inequality 
\begin{equation}
\pi_{R}\left(R_{2n}\right)-\pi_{R}\left(\frac{R_{2n}}{2}\right)\leq2n-n=n\label{27}
\end{equation}
holds for any $n\geq1$. Now by Sondow's Theorem and Rosser's Theorem
we deduce that 
\begin{equation}
4n\ln4n\leq p_{4n}<R_{2n}\leq R_{n}'.\label{28}
\end{equation}

Let us now show the right side of \eqref{a}, namely $R_{n}'<R_{3n}$.
Similarly, it is enough to show that $\pi_{R}\left(R_{3n}\right)-\pi_{R}\left(\frac{R_{3n}}{2}\right)>n$,
that is, $2n>\pi_{R}\left(\frac{R_{3n}}{2}\right)$. This inequality holds
if and only if $\pi_{R}\left(R_{2n}\right)>\pi_{R}\left(\frac{R_{3n}}{2}\right)$,
that is, $2R_{2n}>R_{3n}$. By Sondow's Theorem and Rosser's Theorem we
get
\[
2R_{2n}>2p_{4n}>8n\ln4n.
\]
By \autoref{l4} we have the inequality $8n\ln3n>R_{3n}$ for
any $n\geq5315$. As $8n\ln 4n\geq8n\ln3n$ for all $n\geq1$,
the inequality $2R_{2n}>R_{3n}$ holds for all $n\geq5315$. By computer check we
can see that the right side of the inequality holds also for
$5315>n\geq1$.
\end{proof}
\vskip 10pt

\begin{cor}\label{c4}
For $n>0$, the $n$th derived Ramanujan prime satisfies 
\begin{equation}
p_{4n} < R'_n < p_{9n}.\label{35}
\end{equation}
\end{cor}
\begin{proof}
Use \autoref{t5} together with Sondow's and Laishram's bounds
\begin{equation}
p_{2n} < R_n < p_{3n}.\label{36}
\end{equation}
\end{proof}

Note that the right side of \eqref{35} can be replaced by $R'_n < p_{8n}$ for $n\geq5315$ if we combine \autoref{l4}, \autoref{t5} and Rosser's Theorem.

In \cite{12} Shevelev showed that 
\begin{equation}
\pi_{R}\left(x\right)\sim\frac{\pi\left(x\right)}{2}\sim\frac{x}{2\ln x}\label{29}
\end{equation}
holds following Sondow's $R_{n}\sim p_{2n}$ result. Combining \eqref{29}
with Sondow's method in \cite{15} it is easy to see the following corollary.
Denote by $\pi_{R'}\left(x\right)$ the derived-Ramanujan-prime-counting
function.

\vskip 10pt

\begin{cor}\label{c5}
As $n\rightarrow\infty$ the asymptotic
formula $R_{n}'\sim R_{2n}\sim p_{4n}$ holds, and
given $\varepsilon>0$ there exists $N_{\varepsilon}$ such that $R_{n}'<\left(4+\varepsilon\right)n\ln n$
for $n\geq N_{\varepsilon}$. Moreover
\[
\pi_{R'}\left(x\right)\sim\frac{\pi_{R}\left(x\right)}{2}\sim\frac{\pi\left(x\right)}{4}\sim\frac{x}{4\ln x}.
\]
\end{cor}

\subsubsection*{5. The Number of Primes between $\left(1-\varepsilon\right)x$ and
$x$}

\vskip 10pt

In \cite[Theorem 2.2]{1} N.~Amersi, O.~Beckwith,
S.~J.~Miller, R.~Ronan and J.~Sondow proved that for $c\in\left(0,1\right)$ the number of primes in
the interval $\left(cx,x\right)$ tends to infinity as $x\rightarrow\infty$. We will give another proof of this theorem.

$\hphantom{}$

\begin{thm}\label{t6}
For  any fixed $\epsilon>0$, the number of primes between $\left(1-\varepsilon\right)x$
and~$x$ tends to infinity as $x\rightarrow\infty$.
\end{thm}

\begin{proof} Let $R_{n+1}>x\geq R_{n}$ and therefore $\pi\left(x\right)-\pi\left(\frac{x}{2}\right)\geq n$.
The number of primes between \textit{$\left(1-\varepsilon\right)x$}
and $x$ tends to infinity as $x\rightarrow\infty$ if and only if $\pi\left(\left(1-\varepsilon\right)x\right)-\pi\left(\frac{x}{2}\right)<n-f\left(n\right)$
where $f\left(n\right)$ is a steadily increasing function. But as
\begin{equation}
\pi\left(\left(1-\varepsilon\right)x\right)-\pi\left(\frac{x}{2}\right)<\pi\left(\left(1-\varepsilon\right)R_{n+1}\right)-\pi\left(\frac{R_{n}}{2}\right)\label{30}
\end{equation}
 holds, it is enough to show that 
\begin{equation}
\pi\left(\left(1-\varepsilon\right)R_{n+1}\right)-\pi\left(\frac{R_{n}}{2}\right)<n-f\left(n\right),\label{31}
\end{equation}
or by the equality $n=\pi\left(R_{n}\right)-\pi\left(\frac{R_{n}}{2}\right)$
to prove that $f\left(x\right)$ is not greater than $\pi\left(R_{n}\right)-\pi\left(\left(1-\varepsilon\right)R_{n+1}\right)$.
By Sondow's Theorem \cite{15} we know that for all $\varepsilon>0$
there exists $N\left(\varepsilon\right)$ such that the inequalities
\begin{equation}
\left(2+\varepsilon\right)n\ln n>R_{n}>p_{2n}\label{32}
\end{equation}
 hold for $n>N\left(\varepsilon\right)$. Hence by \autoref{c1} and \eqref{12}
\begin{equation}
\pi\left(R_{n}\right)-\pi\left(\left(1-\varepsilon\right)R_{n+1}\right)>2n-\frac{\left(2-\varepsilon-\varepsilon^{2}\right)\left(n+1\right)\ln\left(n+1\right)}{\ln\left(\left(2-\varepsilon-\varepsilon^{2}\right)\left(n+1\right)\ln\left(n+1\right)\right)-1}\label{33}
\end{equation}
 holds. We can set $f\left(n\right)$ equal to the right side of the inequality
because it tends to infinity as $n\rightarrow\infty$ .
\end{proof}

\subsection*{Acknowledgements}

I would like to thank Dr. Kürşat Aker for his encouragement, helpful counsels and all his efforts. I would also like to especially thank Dr. Jonathan Sondow for his precious comments, reviewing drafts of this paper and most necessary corrections.


\begin{thebibliography}{10}

\bibitem[1]{1}N. Amersi, O. Beckwith, S. J. Miller, R. Ronan and J. Sondow,  Generalized Ramanujan Primes, \texttt{\href{http://arxiv.org/abs/1108.0475}{arXiv:1108.0475}} (2011).

\bibitem[2]{2} P. Dusart, Inégalitiés explicites pour $\psi\left(x\right)$, $\theta\left(x\right)$, $\pi\left(x\right)$ et les nombres premiers, \emph{C. R. Math. Acad. Sci. Soc. R. Can.} \textbf{21(1)} (1999), page 53-59.

\bibitem[3]{3}P. Dusart, Estimates of some functions over primes without R.H., \texttt{\href{http://arxiv.org/abs/1002.0442}{arXiv:1002.0442}} (2010).

\bibitem[4]{4}L. Greenfield and S. Greenfield, Some problems of combinatorial number theory related to Bertrand's Postulate, \emph{J. Integer Seq.} \textbf{1} (1998), Article 98.1.2.

\bibitem[5]{5}S. Laishram, On a conjecture on Ramanujan primes, \emph{Int. J. Number Theory} \textbf{6(8)} (2010), page 1869-1873.

\bibitem[6]{6}J. Nagura, On the interval containing at least one prime number, \emph{Proc. Japan Acad. Ser. A Math. Sci} \textbf{28} (1952), page 177–181.

\bibitem[7]{7}S. Ramanujan, A Proof of Bertrand's Postulate, \emph{J. Indian Math. Soc.} \textbf{XI} (1919), page 181-182.

\bibitem[8]{8}H.-E. Richert, Über Zerfällungen in ungleiche Primzahlen, \emph{Math. Z.} \textbf{52(1)} (1948), page 342-343.

\bibitem[9]{9}J. B. Rosser and L. Schoenfeld, Sharper bounds for the Chebyshev functions $\theta\left(x\right)$ and $\psi\left(x\right)$, \emph{Math. Comp.} \textbf{29} (1975), page 243-269.

\bibitem[10]{10}J. B. Rosser and L. Schoenfeld, Approximate formulas for some functions of Prime numbers, \emph{Illinois J. Math.} \textbf{6} (1962), page 64-94.

\bibitem[11]{11}S. L. Segal, On $\pi\left(x+y\right)\leq\pi\left(x\right)+\pi\left(y\right)$, \emph{Trans. Amer. Math. Soc.} \textbf{104(3)} (1962), page 523-527.

\bibitem[12]{12}V. Shevelev, Ramanujan and Labos primes, their generalizations and classifications of primes, \emph{J. Integer Seq.} \textbf{15} (2012), Article 12.5.4.

\bibitem[13]{13}W. Sierpinski, \emph{Elementary Theory of Numbers}, ICM, pp. 143-144, 1964.

\bibitem[14]{14}N. J. A. Sloane, The On-Line Encyclopedia of Integer Sequences, published electronically at \texttt{\href{http://oeis.org}{oeis.org}} 2011.

\bibitem[15]{15}J. Sondow, Ramanujan Primes and Bertrand's Postulate, \emph{Amer. Math. Monthly} \textbf{116} (2009), page 630-635.	

\bibitem[16]{16}J. Sondow, J. W. Nicholson and T. D. Noe, Ramanujan Primes: Bounds, Runs, Twins and Gaps, \emph{J. Integer Seq.} \textbf{14} (2011), Article 11.6.2.

\end{thebibliography}
\end{document}